\documentclass[sn-mathphys,Numbered]{sn-jnl}% Math and Physical Sciences Reference Style
\usepackage{graphicx}%
\usepackage{multirow}%
\usepackage{amsmath,amssymb,amsfonts}%
\usepackage{amsthm}%
\usepackage{mathrsfs}%
\usepackage[title]{appendix}%
\usepackage{xcolor}%
\usepackage{textcomp}%
\usepackage{manyfoot}%
\usepackage{booktabs}%
\usepackage{algorithm}%
\usepackage{algorithmicx}%
\usepackage{algpseudocode}%
\usepackage{listings}%
\theoremstyle{thmstyleone}%
\newtheorem{theorem}{Theorem}%  meant for continuous numbers
\newtheorem{proposition}[theorem]{Proposition}% 

\theoremstyle{thmstyletwo}%
\newtheorem{example}{Example}%
\newtheorem{corollary}{Corollary}
\theoremstyle{thmstylethree}%
\newtheorem{definition}{Definition}%

\begin{document}

\title[Hyperbolic Spinor Representations of Non-Null Framed Curves]{Hyperbolic Spinor Representations of Non-Null Framed Curves}


\author*[1]{\fnm{Zehra} \sur{ \.I\c{s}bilir}}\email{zehraisbilir@duzce.edu.tr}
\equalcont{These authors contributed equally to this work.}

\author[2]{\fnm{Bahar} \sur{Do\u{g}an Yaz{\i}c{\i}}}\email{bahar.dogan@bilecik.edu.tr}
\equalcont{These authors contributed equally to this work.}

\author[3]{\fnm{Mehmet} \sur{G\"uner}}\email{mguner@sakarya.edu.tr}
\equalcont{These authors contributed equally to this work.}

\affil*[1]{\orgdiv{Department of Mathematics}, \orgname{D\"uzce University}, \orgaddress{\city{D\"uzce}, \postcode{81620}, \country{T{\"u}rkiye}}}

\affil[2]{\orgdiv{Department of Mathematics}, \orgname{Bilecik \c{S}eyh Edebali University}, \orgaddress{\city{Bilecik}, \postcode{11100}, \country{T{\"u}rkiye}}}

\affil[3]{\orgdiv{Department of Mathematics}, \orgname{Sakarya University}, \orgaddress{\city{ Sakarya}, \postcode{54187}, \country{T{\"u}rkiye}}}


\abstract{In this paper, we intend to bring together the hyperbolic spinors, which are useful frameworks from mathematics to physics, and both spacelike and timelike framed curves in Minkowski 3-space $\mathbb{R}_1^3$, which are new type attractive frames and very crucial issue for singularity theory especially. First, we obtain new adapted frames which are called adapted frames for non-null (spacelike and timelike) framed curves in $\mathbb{R}_1^3$. Then, we investigate the hyperbolic spinor representations of non-null framed curves of the general and adapted frames. Also, we find some geometric results and interpretations with respect to them, and we obtain illustrative and numerical examples with figures in order to support the given theorems and results.}




\maketitle

\section{Introduction}
Despite its long history, the theory of curves is one of the most popular, classical, and fundamental subjects of differential geometry, and attracts a great deal of researchers. At the beginning of the discovery of the moving frame, which is called as Frenet-Serret frame, which is found by two distinct researchers Frenet \cite{frenet} and Serret \cite{serret}, was a milestone for differential geometry and  researchers interested in them. A lot of researchers studied the Frenet-Serret frame with respect to several different topics and ongoing. Frenet-Serret frame, which is a fundamental studying area of differential geometry and has been a long time since its investigation, has been an interesting subject, and also has attracted much attention from several researchers.
This frame is convenient to examine, evaluate and interpret the regular curves for lots of geometric properties. In the existing literature, several researchers also found and examined different new types moving frames such as Bishop (type-1, type-2, type-3) \cite{Bishop,Shaker,Syilmaz}, Darboux frame \cite{darboux} and etc.
It is clear that the Frenet-Serret frame is obtained for only regular curves with the condition non-zero curvature. If curves have singular points, then we can not construct the Frenet-Serret frame of them. Afterward, the concept of the framed curve and framed base curve in order to construct the Frenet-Serret frame for non-regular curves are needed \cite{framed}. Framed curves, which were investigated and defined by Honda and Takahashi, are smooth curves with a moving frame that may have singular points \cite{framed}. This new frame also has attracted the notice of several researchers because of the fact that it provides the occasion to obtain a moving frame on curves with singular points. This new interesting frame has contributed to the literature and attracted also plenty of researchers. Framed curves substantially contribute to singularity theory in the existing literature, and the studies interested in this issue are ongoing.

Moreover, Frenet-type framed curves are a special type of framed curves. In order to construct the Frenet frame, if tangent vectors of a non-regular curve vanish at singular points, then a regular spherical curve, which is named as Frenet-type framed base curve, is considered \cite{rectifying,yildiz}. Recently, lots of researchers attached to this topic since this frame have lots of occasion for establishing a moving frame on curves with singular points. In the literature, lots of studies have been done and are ongoing despite of the fact that it has not been a long time since it was discovered. 
The existence conditions of framed curves \cite{existence}, the evolutes of framed curves \cite{evolute}, framed rectifying curves, and framed helices \cite{wang,rectifying,yazici} can be examined. In addition to these, trajectory ruled surfaces related to the framed base curves are scrutinized \cite{yildiz}. These studies studied in the Euclidean space. In addition to these, framed curves also have been studied in the Minkowski $3$-space as spacelike by Li and Pei who introduced the spacelike framed curves in \cite{liandpei}. Also, \"Ozyurt has obtained the special singular curves and properties of them in Minkowski $3$-space in \cite{Cansu}. 

On the other hand, the concept of the spinor is also basic and attractive, and many researchers especially mathematicians and
physicists. P. Ehrenfest is the first person who has used the concept of spinor \cite{Vaz}. Cartan said that spinors satisfy a linear representation of the groups of rotations of a space of any dimension, and also spinors and geometry are very close topics in addition to their relationship with physics \cite{Hladik}. \'E. Cartan who is the first mathematician and studied spinors in a geometrical sense \cite{Cartan}. Since this study includes quite important notions and notations related to the geometry of spinor representation, it is an important study. According to this fundamental study, in the vector space $\mathbb{C}^3$, the set of isotropic vectors establishes a two-dimensional surface in the space $\mathbb{C}^2$.
Conversely, these vectors in $\mathbb{C}^2$ correspond to the same isotropic vectors. Cartan asserted that these vectors are complex as two-dimensional in the space $\mathbb{C}^2$ and he said the spinors including of two complex components related to the vectors in Euclidean 3-space \cite{erisir2,erisir4,Cartan}.
Additionally, Vivarelli is examined the spinors related to the spinors in a geometric sense \cite{Vivarelli}. Some relations between the spinors and quaternions are determined by Vivarelli, and provided the spinor representation of rotations in $E^3$ with the help of the relations between the rotations in $E^3$ and quaternions \cite{Vivarelli,erisir4}.
To get for more detailed information about the spinors, we can refer to the studies \cite{Castillobook,Hladik,Brauer,Cartan,delcastillo,Lounesto,Lounesto2,Vaz, Vivarelli}. Pauli matrices, which have fundamental importance for spinors, were first studied by W. Pauli \cite{paulii}. In the existing literature, lots of studies have been done with respect to spinors.

The study of del Castillo and Barrales \cite{delcastillo} is a milestone for researchers and literature in order to combine the moving frames and spinors. In \cite{delcastillo}, authors study the Frenet-Serret frames and spinors, and they find the spinor representation of Frenet-Serret frame and vectors of this frame. Inspired by this study, lots of researchers also studied the spinors, and different frames and curves. Spinor representations according to the Bishop frame \cite{unal}, Darboux frame \cite{kisi} and Sabban frame \cite{senyurt} are examined. Also, spinor equations are scrutinized related to the different types curves such as successor curves, involute-evolute curves, and Bertrand curves \cite{erisiryeni, erisir1, erisir2}. In addition to these, spinors also are studied with hyperbolic components in Minkowski space  which are called as the hyperbolic spinors \cite{ketenci1,antonuccio}. By using the hyperbolic spinors and frames in the Minkowski space, several studies have been done such as; hyperbolic spinor representations of alternative frame \cite{erisir3}, Frenet frame \cite{ketenci}, Darboux frame \cite{Balci} and Q-frame  \cite{dogan2}. On the other hand, hyperbolic spinors aree examined related to split quaternions for transformations in $\mathbb{R}_1^3$ \cite{tarakcioglu} and Fibonacci spinors are examined \cite{erisir4}. In addition to these, spinor representation of framed Mannheim curves and framed Bertrand curves in Euclidean 3-space are introduced by Yaz{\i}c{\i} et al. \cite{bahar} and \.I\c{s}bilir et al. \cite{zehra}, respectively. We want to also the references with respect to spinors by Aerts et al. \cite{aerts}. Moreover, Eren and Eri\c{s}ir \cite{erenanderisir}, and \"Unal and Yenit\"urk \cite{doganandyeniturk} investigate the spinor representations of $Q$-frame.

We started this study with the question “What are the results when non-null framed curves in Minkowski 3-space and hyperbolic spinors are combined?” and the results made us wonder. This study contains indeed both hyperbolic spinor representations of regular and singular non-null curves in Minkowski 3-space. Since spinors, spinor representations and framed curves are many important topics, in future studies several studies will be completed. Non-null framed curves are a very important framework for singularity theory and recently, lots of studies have been done and are ongoing. We strongly believe that this study will shed light on future studies with respect to the hyperbolic spinor representations and non-null framed curves.

This present study has the following structure as follows. In Section \ref{preliminaries}, we remind some definitions, notations and notions with respect to both hyperbolic spinors and non-null framed curves. Then, in Section \ref{original1}, we construct new adapted frames for spacelike and timelike framed curves in $\mathbb{R}_1^3$. In Section \ref{original2}, we introduce the hyperbolic spinor representations of non-null framed curves in $\mathbb{R}_1^3$ with the help of both the studies in the existing literature and the newly produced adapted frames in this study. Also, we give spinor equations of vectors of non-null framed curves and components of them. Then, we construct two illustrative examples in order to supply the given theorems and results with figures in Section \ref{examples}. Then, we give the conclusions in Section \ref{conclusions}.

\section{Preliminaries}\label{preliminaries}
In this section, we briefly remind some basic notions, notations, and required information with respect to hyperbolic numbers, hyperbolic spinors, notions of Minkowski 3-space $\mathbb{R}^3_1$, and spacelike and timelike framed curves in $\mathbb{R}^3_1$. 

Let the hyperbolic number set is denoted by $\mathbb H$ and any hyperbolic number is written as follows $h=h_1+jh_2\in \mathbb{H}$ where $h_1,h_2\in\mathbb{R}$ and $j^2=1, j\ne\pm1$. Then, the Euler formula for the hyperbolic rotation is given as \cite{sobczyk,erisir3}:
\begin{eqnarray}
e^{j\Theta}=\cosh{\Theta}+j\sinh{\Theta}
\end{eqnarray}
where $\Theta$ is the hyperbolic angle.
For more detailed information related to the hyperbolic number plane can be found in \cite{sobczyk}.

In ${\mathbb{R}_1^3}$, the Lorentzian inner product of any two vectors $\boldsymbol{\omega}=({{\omega}_{1}},{{\omega}_{2}},{{\omega}_{3}})\in{\mathbb{R}_1^3}$ and $\boldsymbol{\varsigma}=({{\varsigma}_{1}},{{\varsigma}_{2}},{{\varsigma}_{3}})\in {\mathbb{R}_1^3}$ is expressed as
$\left\langle \boldsymbol{\omega},\boldsymbol{\varsigma} \right\rangle ={-{\omega}_{1}}{{\varsigma}_{1}}+{{\omega}_{2}}{{\varsigma}_{2}}+{\omega_{3}}{{\varsigma}_{3}}$. If the non-zero vector $\boldsymbol{\omega}$ satisfies $\langle \boldsymbol{\omega},\boldsymbol{\omega}\rangle>0$, $\langle \boldsymbol{\omega},\boldsymbol{\omega}\rangle<0$ and $\langle \boldsymbol{\omega},\boldsymbol{\omega}\rangle=0$, then it is called as a spacelike, timelike or null (lightlike) vector, respectively. In addition to these, the vector product of $\boldsymbol{\omega}$ and $\boldsymbol{\varsigma}$ is given as in $\mathbb{R}_1^3$:
\begin{equation*}
\begin{split}
   \boldsymbol{\omega}\times\boldsymbol{\varsigma}=
\begin{vmatrix}
-\boldsymbol{e}_1&\boldsymbol{e}_2&\boldsymbol{e}_3\\
\omega_1&\omega_2&\omega_3\\
\varsigma_1&\varsigma_2&\varsigma_3
\end{vmatrix}
\end{split}
\end{equation*}
where ${\boldsymbol{e}_n}$ for $n=1,2,3$ are canonical basis of $\mathbb{R}_1^3$ \cite{liandpei}.
Then, the norm in the Minkowski 3-space is defined as follows related to the inner product in Minkowski 3-space: \linebreak $||\boldsymbol{\omega}||=\sqrt{|\langle\boldsymbol\omega,\boldsymbol\omega\rangle|}$ \cite{oneil}.

In addition to these, de Sitter 2-space is determined as \cite{liandpei, Cansu}:
\begin{equation*}
    \mathbb{S}_1^2=\{\boldsymbol{\omega} \in\mathbb{R}_1^3 | \langle\boldsymbol{\omega},\boldsymbol{\omega}    \rangle=1   \}
\end{equation*}
and the hyperbolic 2-space are defined as follows \cite{liandpei, Cansu}:
\begin{equation*}
    \mathbb{H}_0^2=\{\boldsymbol{\omega} \in\mathbb{R}_1^3 | \langle\boldsymbol{\omega},\boldsymbol{\omega}    \rangle=-1   \}.
\end{equation*}

\subsection{Hyperbolic spinors}

Assume that $\flat$ be an $n\times n$ matrix which is defined on the hyperbolic number system $\mathbb{H}$. $\flat^\dagger$ determined as transposing and conjugating of $\flat$, namely $\flat^\dagger=\overline{\flat^t}$, which is an $n\times n$ matrix. If $\flat$ is a Hermitian matrix related to $\mathbb{H}$, then the equation $\flat^t=\flat$ is satisfied. Additionally, if $\flat$ is an anti-Hermitian matrix related to $\mathbb{H}$, then $\flat^t=-\flat$. Let $\flat$ be a Hermitian matrix, the statement $UU^\dagger=U^\dagger U=1$ is valid for $U=e^{j\flat}$. The set of all $n\times n$ type matrices on $\mathbb{H}$ which is satisfied the previous statement construct a group, which is named as hyperbolic unitary group and denoted by $U(n,\mathbb{H})$. Provided that $\det U=1$, then this type group is represented by $SU(n,\mathbb{H})$ \cite{antonuccio,Balci,erisir3}.

The group of all Lorentz transformations in the Minkowski space is named as Lorentz group. The Lorentz group is a subgroup of the Poincar\'e group. Additionally, Poincar\'e group is determined as the group of all isometries in the Minkowski space. The term “orthochronous” is a Lorentz transformation that is kept
in the direction of time. Moreover, the orthochronous Lorentz group is determined as that rigid transformation of Minkowski 3-space which kept both the direction of time and  orientation. Provided that they have the determinant $+1$, then
this subgroup is represented as $SO(1,3)$ \cite{carmeli,Balci,erisir3,ketenci}.

In addition to these, there is a homomorphism between the group $SO(1,3)$, which is the group of the rotation along the origin, and $SU(2,\mathbb{H})$, which is the group of the unitary $2\times 2$ type matrix. While the elements of the group $SU(2,\mathbb{H})$ give a fillip to the hyperbolic spinors, the elements of the group $SO(1,3)$ give a fillip to the vectors with three real components in Minkowski space \cite{sattinger,Balci,erisir3,ketenci}.

A spinor with two hyperbolic components is represented as follows:
\begin{equation*}
    \eta=\begin{pmatrix}
    \eta_1\\
    \eta_2
    \end{pmatrix}
\end{equation*}
by the vectors $a,b,c\in\mathbb{R}_1^3$ such that
\begin{equation}\label{spinor}
\begin{array}{rl}
    a+jb=\eta^t\sigma\eta,\\
    c=-\widehat\eta^t\sigma\eta,
    \end{array}
\end{equation}
where “t” denotes the transposition, $\overline\eta$ is the conjugate of $\eta$, $\widehat\eta$ is the mate of $\eta$. Also, the followings can be written:
\begin{equation*}
   \widehat\eta=-\begin{pmatrix}
   0&1\\
   -1&0
   \end{pmatrix} \overline\eta=-\begin{pmatrix}
    0&1\\
   -1&0
   \end{pmatrix}\begin{pmatrix}
   \overline\eta_1\\
   \overline\eta_2
   \end{pmatrix}=\begin{pmatrix}
   -\overline\eta_2\\
   \overline\eta_1
   \end{pmatrix}.
\end{equation*}
Then, $2\times2$ hyperbolic symmetric matrices which are cartesian components for the vector $\sigma=(\sigma_1,\sigma_2,\sigma_3)$
\begin{equation}\label{matrices}
   \sigma_1= \begin{pmatrix}
    1&0\\
    0&-1
    \end{pmatrix}, \quad   \sigma_2= \begin{pmatrix}
    j&0\\
    0&j
    \end{pmatrix}, \quad  \sigma_3= \begin{pmatrix}
    0&-1\\
    -1&0
    \end{pmatrix}
\end{equation}
can be given \cite{Balci,erisir3,tarakcioglu,ketenci,ketenci1}.
The equation ${\overline{\eta^\prime}}^t\eta^\prime=\overline\eta^t\eta$ is satisfied where $U\in SU(2,\mathbb{H})$ which is an arbitrary matrix and $\eta^\prime=U\eta$. Hence, the norms of the non-null (spacelike and timelike) vectors $a^\prime,b^\prime,c^\prime$ match up with the hyperbolic spinor $\eta^\prime$ are equal to the norms of the non-null vectors $a,b,c$ match up with the hyperbolic spinor $\eta$. So, all elements of the group $SU(2,\mathbb{H})$ establish a transformation which turns into the orthogonal basis $\{a,b,c\}$ of the Minkowski 3-space to the orthogonal basis $\{a^\prime,b^\prime,c^\prime\}$. The correspondence between the hyperbolic spinors and orthogonal bases is two-to-one. That is, $U$
and $-U$, which are any two elements of the group $SU(2,\mathbb{H})$, generate the same ordered triad in $\mathbb{R}_1^3$. Also, the ordered triads $\{a,b,c\}, \{b,c,a\}, \{c,a,b\}$ correspond to different spinors, and the hyperbolic spinors $\eta$ and $-\eta$ correspond to the same ordered orthogonal basis. For the hyperbolic spinors $\eta$ and $\phi$, the following equations can be given: 
\begin{align}
\eta^t\sigma\phi&=\phi^t\sigma\eta,\label{prop2}
\\
    \overline{\eta^t\sigma\phi}&=-\widehat\eta^t\sigma\widehat\phi,\label{prop1}
\\
\widehat{\left(\varrho_1\eta+\varrho_2\phi\right)}&=\overline \varrho_1\widehat\eta+\overline \varrho_2\widehat\phi,\label{prop3}
\end{align}
where $\varrho_1,\varrho_2\in\mathbb{H}$
\cite{Balci,ketenci,erisir3,dogan2,ketenci1}.
Let $\vartheta=(\vartheta_1,\vartheta_2,\vartheta_3)\in\mathbb{H}^3$ be an isotropic vector (namely, length of this vector is zero: $\langle \vartheta,\vartheta \rangle=0$, $\vartheta\ne0$) in $\mathbb{R}_1^3$.
According to the above notions and notations, the following equations can be given:
\begin{eqnarray*}
     \vartheta_1=\eta_1^2-\eta^2_2, \quad  \vartheta_2=j(\eta^2_1+\eta^2_2), \quad \vartheta_3=-2\eta_1\eta_2.
\end{eqnarray*}
Also, the following equations
\begin{eqnarray}\label{xx}
     \eta_1=\pm\sqrt{\frac{\vartheta_1+j\vartheta_2}{2}} \quad \text{and} \quad \eta_2=\pm\sqrt{\frac{-\vartheta_1+j\vartheta_2}{2}}
\end{eqnarray}
can be given. In that case, $||a||=||b||=||c||=\overline\eta^t\eta$. According to the \eqref{spinor} and \eqref{matrices}, the followings
\begin{equation*}
    \begin{array}{rl}
     \vartheta_1=\eta^t\sigma_1\eta=\eta^2_1-\eta^2_2, \quad
       \vartheta_2=\eta^t\sigma_2\eta=j(\eta^2_1+\eta^2_2),\quad
        \vartheta_3=\eta^t\sigma_3\eta=-2\eta_1\eta_2
    \end{array}
\end{equation*}
and
\begin{align} \label{zz}
         a+jb&=\left(\eta^2_1-\eta^2_2,j(\eta^2_1+\eta^2_2),-2\eta_1\eta_2\right), \\
         c&= \left(\eta_1\overline\eta_2+\overline\eta_1\eta_2,j(\eta_1\overline\eta_2-\overline\eta_1\eta_2),\left|\eta_1\right|^2-\left|\eta_2\right|^2\right)
\end{align}
can be written \cite{Balci,ketenci,erisir3,dogan2}.
For more detailed information with respect to the hyperbolic spinor (especially related to hyperbolic spinors and moving frames), we want to refer to the studies \cite{Balci,ketenci,tarakcioglu,erisir3,dogan2}.

\subsection{Spacelike framed curves in $\mathbb{R}^3_1$}

\begin{definition}
	Let $\gamma:I\rightarrow\mathbb{R}^3_1$ be a spacelike curve in $\mathbb{R}^3_1$. For all $s\in I$, if the following conditions are satisfied
	\begin{equation*}
	  \langle \gamma^\prime(s),\nu_1(s) \rangle=0 \quad
\text{and} \quad
	  \langle \gamma^\prime(s),\nu_2(s) \rangle=0
	\end{equation*}
then, the map $(\gamma,\nu_1,\nu_2):I\rightarrow\mathbb{R}^3_1\times\Delta_1$ is named as a spacelike framed curve. Also, $\gamma$ is  named as a base curve of the spacelike framed curve. 
Here, $$\Delta_1=\{(\nu_1,\nu_2)\in \mathbb{S}_1^2\times \mathbb{H}_0^2|\langle \nu_1(s),\nu_2(s) \rangle=0\}$$ or $$\Delta_1=\{(\nu_1,\nu_2)\in \mathbb{H}_0^2\times\mathbb{S}_1^2 |\langle \nu_1(s),\nu_2(s) \rangle=0\}.$$
Then, a spacelike vector field is determined as $\mu(s)=\nu_1(s)\times \nu_2(s)$, which is a smooth function, and $\alpha(s)$ satisfies that $\gamma^\prime(s)=\alpha(s)\mu(s)$. Moreover, the base curve $\gamma(s)$ is singular at the point $s_0$ if and only if $\alpha(s_0)=0$ \cite{liandpei}.
\end{definition}
Furthermore, the Frenet-type formulas are given for a spacelike framed curve as follows \cite{liandpei}:
\begin{equation}\label{framed}
\begin{pmatrix}
\mu^\prime(s)\\
\nu^\prime_1(s)\\
\nu^\prime_2(s)
\end{pmatrix}=\begin{pmatrix}
0&-\delta(s)l_2(s)&\delta(s)l_3(s)\\
l_2(s)&0&l_1(s)\\
l_3(s)&l_1(s)&0
\end{pmatrix}\begin{pmatrix}
\mu(s)\\
\nu_1(s)\\
\nu_2(s)
\end{pmatrix}
\end{equation}
where $\delta(s)=\text{sign}(\nu_1(s))=\langle \nu_1(s),\nu_1(s) \rangle$ and the functions $(\alpha(s),l_1(s),l_2(s),l_3(s))$ are the curvature of the spacelike framed curve. These functions are determined as \cite{liandpei}:
\begin{equation}
\begin{array}{lcl}
\left\{\begin{array}{lcl}
        l_1(s)&=&\langle \nu^\prime_1(s),\nu_2(s) \rangle,\\
        l_2(s)&=&\langle \nu^\prime_1(s),\mu(s) \rangle,\\
        l_3(s)&=&\langle \nu^\prime_2(s),\mu(s) \rangle.
\end{array}\right.
\end{array}
\end{equation}

\subsection{Timelike framed curves in $\mathbb{R}^3_1$}
\begin{definition}
	Let $\gamma:I\rightarrow\mathbb{R}^3_1$ be a timelike curve in $\mathbb{R}^3_1$. For all $s\in I$, if the following conditions are satisfied
	\begin{equation*}
	  \langle \gamma^\prime(s),\nu_1(s) \rangle=0
\quad \text{and} \quad
	  \langle \gamma^\prime(s),\nu_2(s) \rangle=0
	\end{equation*}
then, the map $(\gamma,\nu_1,\nu_2):I\rightarrow\mathbb{R}^3_1\times\Delta_2$ is called as a timelike framed curve. If there exists $\left(\nu_1,\nu_2\right):I\rightarrow\Delta_2$ such that $(\gamma,\nu_1,\nu_2):I\rightarrow\mathbb{R}^3_1\times\Delta_2$ is a timelike framed curve, then the curve $\gamma$ is called as timelike framed type curve. Here, $\Delta_2=\{(\nu_1,\nu_2)|\langle \nu_1(s),\nu_2(s) \rangle=0\}\subset \mathbb{S}_1^2\times \mathbb{S}_1^2$.
Also, $\left\{\mu(s) ,\nu_1(s),\nu_2(s) \right\}$ is a moving frame along the curve $\gamma(s)$ in $\mathbb{R}_1^3$ such that $\mu(s)=\nu_1(s)\times \nu_2(s)\in \mathbb{H}_0^2$, and $\alpha(s)$ satisfies that $\gamma^\prime(s)=\alpha(s)\mu(s)$ \cite{Cansu}.
\end{definition}
Then, the Frenet-type formulas are presented for a timelike framed curve as follows \cite{Cansu}:
\begin{equation}\label{framedtimelike}
\begin{pmatrix}
\mu^\prime(s)\\
\nu^\prime_1(s)\\
\nu^\prime_2(s)
\end{pmatrix}=\begin{pmatrix}
0&l_2(s)&l_3(s)\\
l_2(s)&0&l_1(s)\\
l_3(s)&-l_1(s)&0
\end{pmatrix}\begin{pmatrix}
\mu(s)\\
\nu_1(s)\\
\nu_2(s)
\end{pmatrix}
\end{equation}
where the functions $(\alpha(s),l_1(s),l_2(s),l_3(s))$ are the curvature of the timelike framed curve. Then, these functions are written as \cite{Cansu}:
\begin{equation}
\begin{array}{lcl}
\left\{ \begin{array}{lcl}
l_1(s)&=&\langle \nu^\prime_1(s),\nu_2(s) \rangle,\\
        l_2(s)&=&-\langle \nu^\prime_1(s),\mu(s) \rangle,\\
        l_3(s)&=&-\langle \nu^\prime_2(s),\mu(s) \rangle,\\
      \alpha(s)&=&-\langle \gamma^\prime(s),\mu(s) \rangle.
\end{array}\right.
\end{array}
\end{equation}

\section{The New Adapted Frames for Non-Null Framed Curves}\label{original1}
In this section of this study, we introduce new types adapted frames for both spacelike framed curves and timelike framed curves in Minkowski $3$-space. We want to bring to the literature  these new types adapted frames, and this section makes more comprehensive to our study related to the hyperbolic spinor representations of non-null framed curves.
Thanks to the studies \cite{framed,liandpei,Cansu} and inspired by the study \cite{wang} (by using the method for constructing adapted frame), we obtain this section.
\subsection{A new adapted frame for spacelike framed curves}

\quad

The adapted frame for spacelike framed curves can be constructed and the followings are satisfied:
\begin{eqnarray}\label{AAA}
\left( \begin{matrix}
  \breve{\nu}_1(s) \\
  \breve{\nu}_2(s)
\end{matrix} \right)=\left( \begin{matrix}
   \cosh{{\Theta}}(s) & -\sinh{\Theta}(s) \\
 -\sinh{\Theta}(s)  &\cosh{\Theta}(s)
\end{matrix} \right)\left( \begin{matrix}
  \nu_1 \\
  \nu_2
\end{matrix} \right)
\end{eqnarray}
where $\Theta: I\rightarrow\mathbb{R}$ is a smooth function. Also, $(\gamma,\breve{\nu}_1,\breve{\nu}_2):I\rightarrow\mathbb{R}^3_1\times\Delta_1$ is also a spacelike framed curve and we get: 
\begin{align*}
    \breve{\mu}(s)&=\breve{\nu}_1(s)\times\breve{\nu}_2(s)\\
    &=\left( \cosh{\Theta}(s)\nu_1(s)-\sinh{\Theta}(s)\nu_2(s) \right)\times\left( -\sinh{\Theta}(s)\nu_1(s)+\cosh{\Theta}(s)\nu_2(s) \right)\\&=\nu_1(s)\times\nu_2(s)\\&=\mu(s).
\end{align*}
By straightforward calculations, we have:
\begin{align*}
   \breve{\nu}^\prime_1(s)&= \left( \Theta^\prime(s)-l_1 (s) \right)\sinh{\Theta}(s)\nu_1(s)+\left( -\Theta^\prime(s)+l_1(s) \right)\cosh{\Theta}(s)\nu_2(s)\\&\,\,\,\,\,\,\,+\left( l_2(s)\cosh{\Theta}(s)-l_3(s)\sinh{\Theta}(s)  \right)\mu(s),\\
    \breve{\nu}^\prime_2(s)&= \left( -\Theta^\prime(s)+l_1 (s) \right)\cosh{\Theta}(s)\nu_1(s)+\left( \Theta^\prime(s)-l_1(s) \right)\sinh{\Theta}(s)\nu_2(s)\\&\,\,\,\,\,\,\,+\left( l_3(s)\cosh{\Theta}(s)-l_2(s)\sinh{\Theta}(s)  \right)\mu(s).
\end{align*}
Provided that we take a smooth function $\Theta:I\rightarrow\mathbb{R}$
which holds:
\begin{equation}\label{l1}
\Theta^\prime(s)=l_1(s)
\end{equation}
then we get this triad $\{{\mu}(s),\breve{\nu}_1(s),\breve{\nu}_2(s)\}$ which is an adapted frame along the spacelike framed base curve $\gamma(s)$. Then, we have the following Frenet-Serret-type derivative formulas as follows (Bishop type frame):
\begin{equation}\label{framed2}
\begin{pmatrix}
\mu^\prime(s)\\
\breve{\nu}^\prime_1(s)\\
\breve{\nu}^\prime_2(s)
\end{pmatrix}=\begin{pmatrix}
0&-\delta(s)\breve{l}_2(s)&\delta(s)\breve{l}_3(s)\\
\breve{l}_2(s)&0&0\\
\breve{l}_3(s)&0&0
\end{pmatrix}\begin{pmatrix}
\mu(s)\\
\breve{\nu}_1(s)\\
\breve{\nu}_2(s)
\end{pmatrix}
\end{equation}
where $\breve{l}_2(s)$ and $\breve{l}_3(s)$ are written by
\begin{align}\label{bishopagore}
\left( \begin{matrix}
  \breve{l}_2(s) \\
  \breve{l}_3(s)
\end{matrix} \right)=\left( \begin{matrix}
   \cosh{{\Theta}}(s) & -\sinh{\Theta}(s) \\
-\sinh{\Theta}(s)  &\cosh{\Theta} (s) 
\end{matrix} \right)\left( \begin{matrix}
  l_2(s)\\
  l_3(s)
\end{matrix} \right)
\end{align}
On the other hand, let $\Theta: I\rightarrow\mathbb{R}$ be a  smooth function which holds \linebreak $l_2(s)\sinh\Theta(s)=l_3(s)\cosh\Theta(s)$. Suppose that 
\begin{equation}\label{qq}
l_2(s)=p(s)\cosh\Theta(s) \quad \text{and}\quad l_3(s)=p(s)\sinh\Theta(s)
\end{equation}
then we get as follows:
\begin{align*}
     {\mu}^\prime(s)&= -\delta(s)l_2(s)\nu_1(s)+\delta(s)l_3(s)\nu_2(s)\\
     &=-\delta(s)p(s)\breve{\nu}_1(s),
\end{align*}
\begin{align*}
     \breve{\nu}^\prime_1(s)&= \left( \Theta^\prime(s)-l_1 (s) \right)\sinh{\Theta}(s)\nu_1(s)+\left(-\Theta^\prime(s)+l_1(s) \right)\cosh{\Theta}(s)\nu_2(s)\\&\,\,\,\,\,\,\,+\left( l_2(s)\cosh{\Theta}(s)-l_3(s)\sinh{\Theta}(s)  \right)\mu(s)\\
     &=\left(-\Theta^\prime(s)+l_1(s)  \right)\breve\nu_2(s)+p(s)\mu(s),
\end{align*}
and
\begin{align*}
     \breve{\nu}^\prime_2(s)&= \left(-\Theta^\prime(s)+l_1(s) \right)\cosh{\Theta}(s)\nu_1(s)+\left( \Theta^\prime(s)-l_1(s)  \right)\sinh{\Theta}(s)\nu_2(s)\\&\,\,\,\,\,\,\,+\left( l_3(s)\cosh{\Theta}(s)-l_2(s)\sinh{\Theta}(s) \right)\mu(s)\\
     &=\left(-\Theta^\prime(s)+l_1(s) \right)\breve\nu_1(s).
\end{align*}
In that case, the triad $\{\mu(s),\breve\nu_1(s),\breve\nu_2(s)\}$ become an adapted frame along the spacelike framed curve $\gamma(s)$, and we get the following Frenet-Serret derivative formula:
\begin{equation}\label{framed3}
\begin{pmatrix}
\mu^\prime(s)\\
\breve{\nu}^\prime_1(s)\\
\breve{\nu}^\prime_2(s)
\end{pmatrix}=\begin{pmatrix}
0&-\delta(s)p(s)&0\\
p(s)&0&q(s)\\
0&q(s)&0
\end{pmatrix}\begin{pmatrix}
\mu(s)\\
\breve{\nu}_1(s)\\
\breve{\nu}_2(s)
\end{pmatrix}
\end{equation}
where the vectors $\mu(s), \nu_1(s)$ and $\nu_2(s)$ are called as spacelike generalized tangent vector, spacelike generalized principal normal, and spacelike generalized binormal vector of the spacelike framed curve, respectively. Also,
\begin{equation}\label{qqq}
    p(s)=||\mu^\prime(s)||\ne0 \quad \text{and} \quad q(s)=-\Theta^\prime(s)+l_1(s).
\end{equation} Then, the functions $(p(s), q(s), \alpha(s))$ are called as the curvature of the adapted frame of the spacelike framed curve $\gamma(s)$.

\begin{proposition}
Let $\left(\gamma, \breve\nu_1,\breve\nu_2 \right): I\rightarrow\mathbb{R}_1^3\times\Delta_1$ be a spacelike framed curve. The relation between the first curvature (curvature) $\kappa(s)$ and the second curvature (torsion) $\tau(s)$, and the curvature of the spacelike framed curve $(p(s),q(s),\alpha(s))$ of a regular spacelike curve are written as follows:
\begin{equation*}
    \kappa(s)=\cfrac{\delta(s)p(s)}{\alpha(s)} \quad \text{and} \quad \tau(s)=\cfrac{q(s)}{\alpha(s)}.
\end{equation*}
\end{proposition}
\begin{proof}
The followings can be written easily:
\begin{align*}
    \gamma^\prime(s)=&\alpha(s)\mu(s),\\
    \gamma^{\prime\prime}(s)=&\alpha^\prime(s)\mu(s)-\delta(s)\alpha(s)p(s)\breve\nu_1(s),\\
    \gamma^{\prime\prime\prime}(s)=&\left(\alpha^{\prime\prime}(s)-\delta(s)\alpha(s)p^2(s)\right)\mu(s)-\delta(s)\left(2\alpha^\prime(s)p(s)+\alpha(s)p^\prime(s)\right)\breve\nu_1(s)\\&-\delta(s)\alpha(s)p(s)q(s)\breve\nu_2(s).
    \end{align*}
    Then, we have:
    \begin{align*}
       ||\gamma^\prime(s)||&=|\alpha(s)|,\\
       \gamma^\prime(s)\times\gamma^{\prime\prime}(s)&=-\delta(s)\alpha^2(s)p(s)\breve\nu_2(s),\\
       \det\left(\gamma^\prime(s),\gamma^{\prime\prime}(s),\gamma^{\prime\prime\prime}(s)\right)&=\delta^2(s)\alpha^3(s)p^2(s)q(s).
    \end{align*}
    In that case, we get:
    \begin{align*}
        \kappa(s)&=\cfrac{||\gamma^\prime(s)\times\gamma^{\prime\prime}(s)||}{||\gamma^\prime(s)||^3}=\cfrac{\delta(s)p(s)}{|\alpha(s)|} 
    \end{align*}
    and
        \begin{align*}
        \tau(s)&=\cfrac{\det\left(\gamma^\prime(s),\gamma^{\prime\prime}(s),\gamma^{\prime\prime\prime}(s)\right)}{||\gamma^\prime(s)\times\gamma^{\prime\prime}(s)||^2}=\cfrac{q(s)}{\alpha(s)}.
    \end{align*}
\end{proof}

\subsection{A new adapted frame for timelike framed curves}

We obtain the adapted frame for timelike adapted framed curves and the followings hold:
\begin{eqnarray}\label{BBB}
\left( \begin{matrix}
  \breve{\nu}_1(s) \\
  \breve{\nu}_2(s)
\end{matrix} \right)=\left( \begin{matrix}
   \cos{{\Theta}}(s) & -\sin{\Theta}(s)\\
 \sin{\Theta}(s)  &\cos{\Theta}(s) 
\end{matrix} \right)\left( \begin{matrix}
  \nu_1(s) \\
  \nu_2(s)
\end{matrix} \right)
\end{eqnarray}
where $\Theta(s)$ is a smooth function. Moreover, $(\gamma,\breve{\nu}_1,\breve{\nu}_2):I\rightarrow\mathbb{R}^3_1\times\Delta_2$ is a timelike framed curve and we have:
\begin{align*}
    \breve{\mu}(s)&=\breve{\nu}_1(s)\times\breve{\nu}_2(s)=\left( \cos{\Theta}(s)\nu_1(s)-\sin{\Theta}(s)\nu_2(s) \right)\times\left( \sin{\Theta}(s)\nu_1(s)+\cos{\Theta}(s)\nu_2(s) \right)\\&=\nu_1(s)\times\nu_2(s)\\&=\mu(s).
\end{align*}
Then, if the required calculations are done, we get:
\begin{align*}
   {\breve{\nu}_1}^\prime(s)&= \left( -\Theta^\prime(s)+l_1(s) \right)\sin{\Theta}(s)\nu_1(s)+\left(-\Theta^\prime(s)+l_1(s) \right)\cos{\Theta}(s)\nu_2(s)\\&\,\,\,\,\,\,\,+\left( l_2(s)\cos{\Theta}(s)-l_3(s)\sin{\Theta}(s)  \right)\mu(s),\\
    \breve{\nu}^\prime_2(s)&= \left( \Theta^\prime(s)-l_1(s)  \right)\cos{\Theta}(s)\nu_1(s)+\left( -\Theta^\prime(s)+l_1(s) \right)\sin{\Theta}(s)\nu_2(s)\\&\,\,\,\,\,\,\,+\left( l_2(s)\sin{\Theta}(s)+l_3(s)\cos{\Theta}(s)  \right)\mu(s).
\end{align*}
Taking a smooth function $\Theta:I\rightarrow\mathbb{R}$ which is satisfied $\Theta^\prime(s)=l_1(s)$, then we have this triad $\{{\mu}(s),\breve{\nu}_1(s),\breve{\nu}_2(s)\}$ which is an adapted frame along the timelike framed base curve $\gamma(s)$. Then, we have the following Frenet-Serret type derivative formulas as follows (Bishop type frame):
\begin{equation}\label{framed2timelike}
\begin{pmatrix}
\mu^\prime(s)\\
\breve{\nu}^\prime_1(s)\\
\breve{\nu}^\prime_2(s)
\end{pmatrix}=\begin{pmatrix}
0&\breve{l}_2(s)&\breve{l}_3(s)\\
\breve{l}_2(s)&0&0\\
\breve{l}_3(s)&0&0
\end{pmatrix}\begin{pmatrix}
\mu(s)\\
\breve{\nu}_1(s)\\
\breve{\nu}_2(s)
\end{pmatrix}
\end{equation}
where $\breve{l}_2(s)$ and $\breve{l}_3(s)$ are written by
\begin{align*}
\left( \begin{matrix}
  \breve{l}_2(s) \\
  \breve{l}_3(s) 
\end{matrix} \right)=\left( \begin{matrix}
   \cos{{\Theta}}(s) & -\sin{\Theta}(s)\\
 \sin{\Theta}(s)  &\cos{\Theta}(s)  
\end{matrix} \right)\left( \begin{matrix}
  l_2(s)\\
  l_3(s)
\end{matrix} \right)
\end{align*}
Then, let $\Theta: I\rightarrow\mathbb{R}$ be a  smooth function which holds $l_2(s)\sin\Theta(s)=-l_3(s)\cos\Theta(s)$. Suppose that $l_2(s)=p(s)\cos\Theta(s)$ and $l_3(s)=-p(s)\sin\Theta(s)$, then we obtain as follows:
\begin{align*}
     {\mu}^\prime(s)&= l_2(s)\nu_1(s)+l_3(s)\nu_2(s)\\
     &=p(s)\breve{\nu}_1(s),
\end{align*}
\begin{align*}
    \breve{\nu}^\prime_1(s)&= \left( -\Theta^\prime(s)+l_1(s) \right)\sin{\Theta}(s)\nu_1(s)+\left(-\Theta^\prime(s)+l_1(s) \right)\cos{\Theta}(s)\nu_2(s)\\&\,\,\,\,\,\,\,+\left( l_2(s)\cos{\Theta}(s)-l_3(s)\sin{\Theta}(s)  \right)\mu(s)\\
     &=\left(-\Theta^\prime(s)+l_1(s)  \right)\breve\nu_2(s)+p(s)\mu(s),
\end{align*}
and

\begin{align*}
    \breve{\nu}^\prime_2(s)=&\left( \Theta^\prime(s)-l_1(s) \right)\cos{\Theta}(s)\nu_1(s)+\left( -\Theta^\prime(s)+l_1(s) \right)\sin{\Theta}(s)\nu_2(s)\\&+\left( l_2(s)\sin{\Theta}(s)+l_3(s)\cos{\Theta}(s)  \right)\mu(s)\\
     =&\left( \Theta^\prime(s)-l_1(s)  \right)\breve\nu_1.
\end{align*}
Then, the triad $\{\mu(s),\breve\nu_1(s),\breve\nu_2(s)\}$ construct an adapted frame along the timelike curve $\gamma(s)$, and we have the following Frenet-Serret derivative formula:

\begin{equation}\label{framedtimelike3}
\begin{pmatrix}
\mu^\prime(s)\\
\breve{\nu}^\prime_1(s)\\
\breve{\nu}^\prime_2(s)
\end{pmatrix}=\begin{pmatrix}
0&p(s)&0\\
p(s)&0&q(s)\\
0&-q(s)&0
\end{pmatrix}\begin{pmatrix}
\mu(s)\\
\breve{\nu}_1(s)\\
\breve{\nu}_2(s)
\end{pmatrix}
\end{equation}
where the vectors $\mu(s), \nu_1(s)$ and $\nu_2(s)$ are called as timelike generalized tangent vector, timelike generalized principal normal, and timelike generalized binormal vector of the timelike framed curve, respectively. Also, $p(s)=||\mu^\prime(s)||\ne0$ and \linebreak $q(s)=-\Theta^\prime(s)+l_1(s)$. Then, the functions $(p(s), q(s), \alpha(s))$ are called as the curvature of the adapted frame of the timelike framed curve $\gamma(s)$.

\begin{proposition}
Let $\left(\gamma, \breve\nu_1,\breve\nu_2 \right): I\rightarrow\mathbb{R}_1^3\times\Delta_2$ be a timelike framed curve. The relation between the first curvature (curvature) $\kappa(s)$ and the second curvature (torsion) $\tau(s)$, and the curvature of the timelike framed curve $(p(s),q(s),\alpha(s))$ of a regular timelike curve are expressed as follows:
\begin{equation*}
    \kappa(s)=\cfrac{p(s)}{\alpha(s)} \quad \text{and} \quad \tau(s)=\cfrac{q(s)}{\alpha(s)}.
\end{equation*}
\end{proposition}
\begin{proof}
The followings can be obtained:
\begin{align*}
    \gamma^\prime(s)&=\alpha(s)\mu(s),\\
    \gamma^{\prime\prime}(s)&=\alpha^\prime(s)\mu(s)+\alpha(s)p(s)\breve\nu_1(s),\\
    \gamma^{\prime\prime\prime}(s)&=\left(\alpha^{\prime\prime}(s)+\alpha(s)p^2(s)\right)\mu(s)+\left(2\alpha^\prime(s)p(s)+\alpha(s)p^\prime(s)\right)\breve\nu_1(s)+\alpha(s)p(s)q(s)\breve\nu_2(s).
    \end{align*}
    In that case, we get:
    \begin{align*}
       ||\gamma^\prime(s)||&=|\alpha(s)|,\\
       \gamma^\prime(s)\times\gamma^{\prime\prime}(s)&=\alpha^2(s)p(s)\breve\nu_2(s),\\
       \det\left(\gamma^\prime(s),\gamma^{\prime\prime}(s),\gamma^{\prime\prime\prime}(s)\right)&=\alpha^3(s)p^2(s)q(s).
    \end{align*}
    Then, we have:
    \begin{align*}
        \kappa(s)&=\cfrac{||\gamma^\prime(s)\times\gamma^{\prime\prime}(s)||}{||\gamma^\prime(s)||^3}=\cfrac{p(s)}{|\alpha(s)|} 
    \end{align*}
    and
        \begin{align*}
        \tau(s)&=\cfrac{\det\left(\gamma^\prime(s),\gamma^{\prime\prime}(s),\gamma^{\prime\prime\prime}(s)\right)}{||\gamma^\prime(s)\times\gamma^{\prime\prime}(s)||^2}=\cfrac{q(s)}{\alpha(s)}.
    \end{align*}
\end{proof}

\section{Hyperbolic spinor representations of non-null framed curves}\label{original2}
In this section, we introduce the hyperbolic spinor representations of spacelike and timelike framed curves in Minkowski 3-space $\mathbb{R}_1^3$. We organize by seperating this section as two subparts. First of all, hyperbolic spinor representations of spacelike framed curves are determined and examined, then hyperbolic spinor representations of timelike framed curves are given. Also, we obtain some geometric interpretations with respect to them.

It should be noted that, we do not use the parameter “$s$” all of the equations in the definitions, theorems and conclusions in this section for the sake of the brevity.

\subsection{Hyperbolic spinor representations of spacelike framed curves}

\quad 

In this part of this study, we investigate and scrutinize the hyperbolic spinor representations of spacelike framed curves in $\mathbb{R}_1^3$.

\begin{definition}\label{definition4.1} Let
$\left(\gamma,  {\nu}_1,  {\nu}_2  \right): I\rightarrow \mathbb{R}_1^3\times\Delta_1$ be a spacelike framed curve and the hyperbolic spinor $\phi$ represents the triad $\{ \nu_1,  {\nu_2}, {\mu}\}$. Then, the hyperbolic spinor representations of the spacelike framed curve are defined as follows:
\begin{align}
    \nu_1+j\nu_2&=\phi^t\sigma\phi,\label{eq1}
\\
\mu&=-\widehat{\phi}^t\sigma\phi,\label{eq2}
\end{align}
where $\overline{\phi}^t\phi=1$.
\end{definition}

\begin{theorem}\label{thmnew}
Let $\left(\gamma, \nu_1,\nu_2 \right): I\rightarrow\mathbb{R}_1^3\times\Delta_1$ be a spacelike framed curve and the hyperbolic spinor $\phi$ represents the triad $\{\nu_1,\nu_2, \mu\}$. The single spinor equation that includes the curvatures of the spacelike framed curve is written as:
\begin{equation}\label{new}
    \cfrac{d\phi}{ds}=\cfrac{1}{2}\left[jl_1\phi-\left(l_2+jl_3\right)\widehat\phi\right].
\end{equation}
\end{theorem}

\begin{proof}
By taking the derivative of the equation \eqref{eq1} with respect to the parameter $s$, then we have:
\begin{equation}\label{dif2}
    \cfrac{d\nu_1}{ds}+j\cfrac{d  {\nu_2}}{ds}=\left(\cfrac{d\phi}{ds} \right)^t
\sigma\phi+\phi^t\sigma\cfrac{d\phi}{ds}.
\end{equation}
Since the system $\{\phi,\widehat\phi\}$ is a basis for the hyperbolic spinor $\phi$, then $\cfrac{d\phi}{ds}$ can be written as follows:
\begin{equation}\label{baz2}
    \cfrac{d\phi}{ds}=\xi\phi+\chi\widehat\phi
\end{equation}
where $\xi$ and $\chi$ are any two hyperbolic valued functions. By using the  equations \eqref{framed}, \eqref{dif2} and \eqref{baz2}, we have:
\begin{align*}
        l_1{\nu}_2+ l_2{\mu} +j\left(l_1{\nu}_1+ l_3 {\mu}   \right)&=\left(\xi\phi+\chi \widehat\phi \right)^t\sigma\phi+\phi^t\sigma\left(\xi\phi+\chi\widehat\phi   \right)\\
         &=\xi\phi^t\sigma\phi+\chi{\widehat\phi}^t\sigma\phi+\xi\phi^t\sigma\phi+\chi\phi^t\sigma\widehat\phi.
\end{align*}
Then, with the help of the equations \eqref{prop2}, \eqref{eq1} and \eqref{eq2}, we get:
\begin{align*}
     jl_1\left({\nu}_1+j{\nu}_2\right)+\left({l}_2+jl_3\right){\mu} &=2\xi\left( \nu_1+j{\nu}_2\right)-2\chi{\mu}.
\end{align*}
In that case, we obtain:
\begin{equation}\label{j}
 \xi=\cfrac{jl_1}{2}\quad \text{and} \quad
\chi=-\cfrac{l_2+jl_3}{2}.   
\end{equation}
By substituting the equation \eqref{j} in the equation \eqref{baz2}, then we get the equation \eqref{new}.
\end{proof}

\begin{theorem}\label{thmnew2} Let $\left(\gamma, \nu_1,\nu_2 \right): I\rightarrow\mathbb{R}_1^3\times\Delta_1$ be a spacelike framed curve and the hyperbolic spinor $\phi$ represents the triad $\{\nu_1,\nu_2, \mu\}$. The hyperbolic spinor representations of vectors of the spacelike framed curve are given as:
\begin{align}
   {\mu} &=-\widehat{\phi}^t\sigma\phi, \label{a1}\\
    \nu_1&=\cfrac{1}{2}\left(\phi^t\sigma\phi-{\widehat{\phi}^t\sigma\widehat{\phi}}\right),\label{a2}\\
      {\nu_2}&=\cfrac{j}{2}\left(\phi^t\sigma\phi+\widehat{\phi}^t\sigma\widehat\phi\right).\label{a3}
\end{align}
\end{theorem}

\begin{proof}
Let the hyperbolic spinor $\phi$ corresponds to the triad $\{\nu_1,\nu_2, \mu\}$ of the spacelike framed curve $\left(\gamma, \nu_1,\nu_2 \right)$. According to the equation \eqref{eq1}, we get $\nu_1=Re(\phi^t\sigma\phi)$ and $\nu_2=Im(\phi^t\sigma\phi)$. We have already $\mu=-\widehat{\phi}^t\sigma\phi$ from the equation \eqref{eq2}. Also, by using the well-known properties of hyperbolic numbers: $Re(\rho)=\frac{1}{2}\left(\rho+\overline \rho  \right)$ and $Im(\rho)=\frac{j}{2}\left(\rho-\overline \rho  \right)$ for every $\rho\in\mathbb{H}$, we can write as follows:
    \begin{equation*}
        \nu_1=\cfrac{1}{2}\left(\phi^t\sigma\phi+\overline{\phi^t\sigma\phi}\right),
    \end{equation*}
      \begin{equation*}
       \nu_2=\cfrac{j}{2}\left(\phi^t\sigma\phi-\overline{\phi^t\sigma\phi}\right).
    \end{equation*}
According to the last two equations and by using the equation\eqref{prop1}, we obtain:
    \begin{equation*}
        \nu_1=\cfrac{1}{2}\left(\phi^t\sigma\phi-{\widehat{\phi}^t\sigma\widehat{\phi}}\right),
    \end{equation*}
      \begin{equation*}
        \nu_2=\cfrac{j}{2}\left(\phi^t\sigma\phi+\widehat{\phi}^t\sigma\widehat\phi\right),
    \end{equation*}
and complete the proof.
\end{proof}
\begin{corollary}\label{corollary4.1}
   Let $\left(\gamma, \nu_1,\nu_2 \right): I\rightarrow\mathbb{R}_1^3\times\Delta_1$ be a spacelike framed curve and the hyperbolic spinor $\phi$ represents the triad $\{\nu_1,\nu_2, \mu\}$. Then the hyperbolic spinor components for the spacelike framed vectors are given as:
    \begin{align*}
        \nu_1&= \cfrac{1}{2}\left(\phi^2_1 -\phi^2_2+\overline\phi^2_1-\overline\phi^2_2, j\left(\phi^2_1+\phi^2_2-\overline\phi^2_1-\overline\phi^2_2   \right), -2\left(\phi_1\phi_2+\overline\phi_1\overline\phi_2\right) \right),
       \\
             \nu_2&= \cfrac{j}{2}\left(\phi^2_1 -\phi^2_2-\overline\phi^2_1+\overline\phi^2_2, j\left(\phi^2_1+\phi^2_2+\overline\phi^2_1+\overline\phi^2_2   \right), 2\left(\overline\phi_1\overline\phi_2-\phi_1\phi_2\right) \right),
            \\
             \mu&= \left(\phi_1\overline\phi_2 +\overline\phi_1\phi_2,j\left(\phi_1\overline\phi_2-\overline\phi_1\phi_2\right),|\phi_1|^2-|\phi_2|^2   \right).
    \end{align*}
\end{corollary}
\begin{proof}
 Let $\left(\gamma, \nu_1,\nu_2 \right): I\rightarrow\mathbb{R}_1^3\times\Delta_1$ be a spacelike framed curve and the hyperbolic spinor $\phi$ corresponds to the triad $\{\nu_1,\nu_2, \mu\}$. Then, for $\phi^t\sigma\phi$, we can write:
 \begin{equation}\label{1}
 \begin{array}{lcl}\left\{
 \begin{array}{lcl}
     \phi^t\sigma_1\phi&=&\begin{pmatrix}
     \phi_1&\phi_2
     \end{pmatrix}\begin{pmatrix}
     1&0\\
     0&-1
     \end{pmatrix}\begin{pmatrix}
     \phi_1\\
     \phi_2
     \end{pmatrix}=\phi_1^2-\phi_2^2,\\
     \phi^t\sigma_2\phi&=&\begin{pmatrix}
     \phi_1&\phi_2
     \end{pmatrix}\begin{pmatrix}
     j&0\\
     0&j
     \end{pmatrix}\begin{pmatrix}
     \phi_1\\
     \phi_2
     \end{pmatrix}=j(\phi_1^2+\phi_2^2),\\
     \phi^t\sigma_3\phi&=&\begin{pmatrix}
     \phi_1&\phi_2
     \end{pmatrix}\begin{pmatrix}
     0&-1\\
     -1&0
     \end{pmatrix}\begin{pmatrix}
     \phi_1\\
     \phi_2
     \end{pmatrix}=-2\phi_1\phi_2.
     \end{array}
     \right.
      \end{array}
 \end{equation}
 Also, for $\widehat{\phi}^t\sigma\widehat{\phi}$, we have:
 \begin{equation}\label{2}
 \begin{array}{lcl}\left\{
 \begin{array}{lcl}
     \widehat{\phi}^t\sigma_1\widehat{\phi}&=&\begin{pmatrix}
     -\overline{\phi}_2&\overline{\phi}_1
     \end{pmatrix}\begin{pmatrix}
     1&0\\
     0&-1
     \end{pmatrix}\begin{pmatrix}
     -\overline{\phi}_2\\
     \overline{\phi}_1
     \end{pmatrix}=\overline{\phi}_2^2-\overline{\phi}_1^2,\\
      \widehat{\phi}^t\sigma_2\widehat{\phi}&=&\begin{pmatrix}
     -\overline{\phi}_2&\overline{\phi}_1
     \end{pmatrix}\begin{pmatrix}
     j&0\\
     0&j
     \end{pmatrix}\begin{pmatrix}
     -\overline{\phi}_2\\
     \overline{\phi}_1
     \end{pmatrix}=j\left(\overline{\phi}_2^2+\overline{\phi}_1^2\right),\\
      \widehat{\phi}^t\sigma_3\widehat{\phi}&=&\begin{pmatrix}
     -\overline{\phi}_2&\overline{\phi}_1
     \end{pmatrix}\begin{pmatrix}
     0&-1\\
     -1&0
     \end{pmatrix}\begin{pmatrix}
     -\overline{\phi}_2\\
     \overline{\phi}_1
     \end{pmatrix}=2\overline{\phi}_1\overline{\phi}_2.
     \end{array}
     \right.
      \end{array}
 \end{equation}
Then, for $\widehat{\phi}^t\sigma\phi$, we get:
 \begin{equation}\label{3}
 \begin{array}{lcl}\left\{
 \begin{array}{lcl}
     \widehat{\phi}^t\sigma_1\phi&=&\begin{pmatrix}
     -\overline{\phi}_2&\overline{\phi}_1
     \end{pmatrix}\begin{pmatrix}
     1&0\\
     0&-1
     \end{pmatrix}\begin{pmatrix}
     \phi_1\\
     \phi_2
     \end{pmatrix}=-\overline{\phi}_2\phi_1-\overline{\phi}_1\phi_2,\\
     \widehat{\phi}^t\sigma_2\phi&=&\begin{pmatrix}
     -\overline{\phi}_2&\overline{\phi}_1
     \end{pmatrix}\begin{pmatrix}
     j&0\\
     0&j
     \end{pmatrix}\begin{pmatrix}
     \phi_1\\
     \phi_2
     \end{pmatrix}=i\left(-\overline{\phi}_2\phi_1+\overline{\phi}_1\phi_2\right),\\
     \widehat{\phi}^t\sigma_3\phi&=&\begin{pmatrix}
     -\overline{\phi}_2&\overline{\phi}_1
     \end{pmatrix}\begin{pmatrix}
     0&-1\\
     -1&0
     \end{pmatrix}\begin{pmatrix}
     \phi_1\\
     \phi_2
     \end{pmatrix}=-|\overline{\phi}_1|^2+|\overline{\phi}_2|^2.
     \end{array}\right.
      \end{array}
 \end{equation}
 If we substitute the equations \eqref{1}, \eqref{2} and \eqref{3} in the equations \eqref{a1}, \eqref{a2} and \eqref{a3}, we have desired result.
\end{proof}
From the Definition \ref{definition4.1} to Corollary \ref{corollary4.1}, all of the notions and notations in the definition, theorems, and corollaries can be obtained easily for the adapted frame for spacelike framed curve, but for the sake of the brevity and  since they are very clear, we do not give them. 
However, we want to present also the following definition:
\begin{definition}\label{definition4.2} Let
$\left(\gamma,  \breve{\nu}_1,  \breve{\nu}_2  \right): I\rightarrow \mathbb{R}_1^3\times\Delta_1$ be a spacelike framed curve and the hyperbolic spinor $\phi$ represents the triad $\{ \breve\nu_1,  {\breve\nu_2}, {\mu}\}$. Then, the hyperbolic spinor representations of the adapted framed frame along the spacelike framed curve are defined as follows:
\begin{align}
    \breve\nu_1+j\breve\nu_2&=\Upsilon^t\sigma\Upsilon,\label{eqq1}
\\
\mu&=-\widehat{\Upsilon}^t\sigma\Upsilon,\label{eqq2}
\end{align}
where $\overline{\Upsilon}^t\Upsilon=1$.
\end{definition}
Also, the following single spinor equations with respect to the adapted frame for spacelike framed curves can be seen:
\begin{enumerate}
    \item Let $\left(\gamma, \breve\nu_1,\breve\nu_2 \right): I\rightarrow\mathbb{R}_1^3\times\Delta_1$ be a spacelike framed curve and the hyperbolic spinor $\Upsilon$ represents the triad $\{\breve\nu_1,\breve\nu_2, \mu\}$. Then, we have (according to the Bishop type frame): \begin{equation}
    \cfrac{d\Upsilon}{ds}=-\cfrac{1}{2}\left(\breve l_2+j\breve l_3\right)\widehat\Upsilon. 
\end{equation}
    \item Let $\left(\gamma, \breve\nu_1,\breve\nu_2 \right): I\rightarrow\mathbb{R}_1^3\times\Delta_1$ be a spacelike framed curve and the hyperbolic spinor $\Upsilon$ represents the triad $\{\breve\nu_1,\breve\nu_2, \mu\}$. Then, we have (according to the Frenet-Serret type frame): \begin{equation}
    \cfrac{d\Upsilon}{ds}=\cfrac{1}{2}\left(jq\Upsilon-p\widehat\Upsilon \right).
\end{equation}
\end{enumerate}

Now, in the following Theorem \ref{relation}, let us construct the spinor relations between adapted frame and general frame of the spacelike framed curves:

\begin{theorem}\label{relation}
 Let $\left(\gamma, \nu_1,\nu_2 \right), \left(\gamma, \breve\nu_1,\breve\nu_2 \right): I\rightarrow\mathbb{R}_1^3\times\Delta_1$ be spacelike framed curves and the hyperbolic spinors $\phi$ and $\Upsilon$ represent the triads $\{\nu_1,\nu_2, \mu\}$ and $\{\breve\nu_1,\breve\nu_2,\mu   \}$, respectively.
The following spinor relations are given:
\begin{align}
   \phi^t\sigma\phi&=e^{j\Theta}\left(\Upsilon^t\sigma\Upsilon  \right),\label{ed}\\
   \mu&=\mu\nonumber,
\end{align}
where $\mu=\breve\mu$ and the hyperbolic angle between the vectors $\nu_1$ and $\breve\nu_1$ is $\Theta$.
 \end{theorem}
 \begin{proof} By using the equation \eqref{AAA}, we have:
 \begin{align}
    \nu_1+j\nu_2&=\left( \cosh{\Theta}\breve{\nu}_1+\sinh{\Theta}\breve{\nu}_2  \right)+j\left( \sinh{\Theta}\breve{\nu}_1+\cosh{\Theta}\breve{\nu}_2  \right)\nonumber\\
    &=\left(\breve{\nu}_1+j\breve{\nu}_2 \right)\cosh{\Theta}+j\left( \breve{\nu}_1+j\breve{\nu}_2 \right)\sinh{\Theta}\nonumber\\
    &=\left( \breve{\nu}_1+j\breve{\nu}_2 \right)\left( \cosh\Theta+j\sinh\Theta  \right)\nonumber\\
    &=e^{j\Theta}\left( \breve{\nu}_1+j\breve{\nu}  _2 \right).\label{ee}
\end{align}
Then, via the equation \eqref{eq1} and \eqref{eqq1}, we have the desired result.
 \end{proof}

\begin{theorem} Let $\left(\gamma, \nu_1,\nu_2 \right), \left(\gamma, \breve\nu_1,\breve\nu_2 \right): I\rightarrow\mathbb{R}_1^3\times\Delta_1$ be spacelike framed curves and the hyperbolic spinors $\phi$ and $\Upsilon$ represent the triads $\{\nu_1,\nu_2, \mu\}$ and $\{\breve\nu_1,\breve\nu_2,\mu   \}$, respectively. The following relation between these spinors $\phi$ and $\Upsilon$ holds:
    \begin{equation}\label{3.12}
       \phi=\pm e^{j\frac{\Theta}{2}}\Upsilon. 
    \end{equation}
\end{theorem}
\begin{proof}
Via the equations \eqref{zz} and \eqref{ee}, we have:
\begin{align*} \phi^t\sigma\phi&=\left(\phi^2_1-\phi^2_2,j\left( \phi^2_1+\phi^2_2 \right), -2\phi_1\phi_2 \right),\\
\Upsilon^t\sigma\Upsilon&=\left(\Upsilon^2_1-\Upsilon^2_2,j\left( \Upsilon^2_1+\Upsilon^2_2 \right), -2\Upsilon_1\Upsilon_2 \right),
\end{align*}
and then we get $\phi^2_1=e^{j\Theta}\Upsilon^2_1$ and $\phi^2_2=e^{j\Theta}\Upsilon^2_2$. 
 Thus, $\phi_1=\pm e^{j\frac{\Theta}{2}}\phi_1$ and $\phi_2=\pm e^{j\frac{\Theta}{2}}\phi_2$ can be obtained. The spinor $\phi$ and $-\phi$ correspond to the triad $\{\nu_1,\nu_2,\mu   \}$, and the spinor $\Upsilon$ and $-\Upsilon$ correspond to the triad $\{\breve\nu_1,\breve\nu_2,\mu   \}$. Hence, we can write $\phi=\pm e^{j\frac{\Theta}{2}}\Upsilon$.
\end{proof}

\begin{corollary} Let $\left(\gamma, \nu_1,\nu_2 \right), \left(\gamma, \breve\nu_1,\breve\nu_2 \right): I\rightarrow\mathbb{R}_1^3\times\Delta_1$ be spacelike framed curves and the hyperbolic spinors $\phi$ and $\Upsilon$ represent the triads $\{\nu_1,\nu_2, \mu\}$ and $\{\breve\nu_1,\breve\nu_2,\mu   \}$, respectively. Then, the angle between the spinors $\phi$ and $\Upsilon$ is $\Theta/2$.
\end{corollary}

\begin{theorem} Let $\left(\gamma, \nu_1,\nu_2 \right), \left(\gamma, \breve\nu_1,\breve\nu_2 \right): I\rightarrow\mathbb{R}_1^3\times\Delta_1$ be spacelike framed curves and the hyperbolic spinors $\phi$ and $\Upsilon$ represents the triads $\{\nu_1,\nu_2, \mu\}$ and $\{\breve\nu_1,\breve\nu_2,\mu   \}$, respectively. The following relation between the spinors $\phi$ and $\Upsilon$ is satisfied:
\begin{equation}\label{4.17}
    \widehat\phi=\pm e^{-j\frac{\Theta}{2}}\widehat\Upsilon.
\end{equation}
\end{theorem}
\begin{proof}By taking the mate of both sides of the equation \eqref{3.12}, we have
$
    \widehat \phi=\widehat{\pm e^{j\frac{\Theta}{2}}\Upsilon}.
$
With the help of the equation \eqref{prop3}, we get  $\widehat\phi=\pm e^{-j\frac{\Theta}{2}}\widehat\Upsilon$.
\end{proof}

\begin{corollary} Let $\left(\gamma, \nu_1,\nu_2 \right), \left(\gamma, \breve\nu_1,\breve\nu_2 \right): I\rightarrow\mathbb{R}_1^3\times\Delta_1$ be spacelike framed curves and the hyperbolic spinors $\phi$ and $\Upsilon$ represent the triads $\{\nu_1,\nu_2, \mu\}$ and $\{\breve\nu_1,\breve\nu_2,\mu   \}$, respectively. While the spinor $\phi$ makes a rotation with the angle $\Theta/2$ to the spinor $\Upsilon$, the spinor $\widehat\phi$ makes an opposite rotation with the same angle to the spinor $\widehat\Upsilon$.
\end{corollary}

\begin{theorem}  Let $\left(\gamma, \nu_1,\nu_2 \right), \left(\gamma, \breve\nu_1,\breve\nu_2 \right): I\rightarrow\mathbb{R}_1^3\times\Delta_1$ be spacelike framed curves and the hyperbolic spinors $\phi$ and $\Upsilon$ represent the triads $\{\nu_1,\nu_2, \mu\}$ and $\{\breve\nu_1,\breve\nu_2,\mu   \}$, respectively.
The derivative of the spinor $\phi$ can be written by using the curvatures with respect to the adapted spacelike framed frame (according to the Bishop type frame).
\begin{equation}\label{4.18}
    \cfrac{d\phi}{ds}=\cfrac{1}{2}\left[j\Theta^\prime\phi-\left(\breve{l}_2+j\breve{l}_3 \right)e^{j\Theta}\widehat\phi\right].
\end{equation}
\end{theorem}
\begin{proof}
From the equations \eqref{new}, \eqref{l1} and \eqref{bishopagore}, we have:
\begin{align}
     \cfrac{d\phi}{ds}&=\cfrac{1}{2}\left[jl_1\phi-\left(l_2+jl_3\right)\widehat\phi\right]\nonumber\\
     &=\cfrac{1}{2}\left[j\Theta^\prime\phi-\left(\left( \cosh{\Theta}\breve{l}_2+\sinh{\Theta}\breve{l}_3  \right)+j\left( \sinh{\Theta}\breve{l}_2+\cosh{\Theta}\breve{l}_3  \right)\right)\widehat\phi \right]\nonumber\\
    & =\cfrac{1}{2}\left[j\Theta^\prime\phi-\left(\breve{l}_2+j\breve{l}_3 \right)e^{j\Theta}\widehat\phi\right]\nonumber.
\end{align}
\end{proof}

\begin{corollary}  Let $\left(\gamma, \nu_1,\nu_2 \right), \left(\gamma, \breve\nu_1,\breve\nu_2 \right): I\rightarrow\mathbb{R}_1^3\times\Delta_1$ be spacelike framed curves and the hyperbolic spinors $\phi$ and $\Upsilon$ represent the triads $\{\nu_1,\nu_2, \mu\}$ and $\{\breve\nu_1,\breve\nu_2,\mu   \}$, respectively. The angle between $\cfrac{d\phi}{ds}$ and $\widehat{\phi}$ is $\Theta$ on condition that $\Theta^\prime=0$.
\end{corollary}

\begin{theorem} Let $\left(\gamma, \nu_1,\nu_2 \right), \left(\gamma, \breve\nu_1,\breve\nu_2 \right): I\rightarrow\mathbb{R}_1^3\times\Delta_1$ be spacelike framed curves and the hyperbolic spinors $\phi$ and $\Upsilon$ represent the triads $\{\nu_1,\nu_2, \mu\}$ and $\{\breve\nu_1,\breve\nu_2,\mu   \}$, respectively. Then the following equation is satisfied:
\begin{equation*}
    \cfrac{d\phi}{ds}=\pm\cfrac{1}{2}e^{j\frac{\Theta}{2}}\left[j\Theta^\prime\Upsilon-\left(\breve{l}_2+j\breve{l}_3 \right)\widehat\Upsilon\right].
\end{equation*}
\end{theorem}
\begin{proof}
By using the equation \eqref{3.12}, \eqref{4.17} and \eqref{4.18}, we have the desired result.
\end{proof}

\begin{corollary}
 Let $\left(\gamma, \nu_1,\nu_2 \right), \left(\gamma, \breve\nu_1,\breve\nu_2 \right): I\rightarrow\mathbb{R}_1^3\times\Delta_1$ be spacelike framed curves and the hyperbolic spinors $\phi$ and $\Upsilon$ represent the triads $\{\nu_1,\nu_2, \mu\}$ and $\{\breve\nu_1,\breve\nu_2,\mu   \}$, respectively. The angle between $\cfrac{d\phi}{ds}$ and $\widehat{\Upsilon}$ is $\Theta/2$ on condition that $\Theta^\prime=0$.
\end{corollary}

\begin{theorem} Let $\left(\gamma, \nu_1,\nu_2 \right): I\rightarrow\mathbb{R}_1^3\times\Delta_1$ be a spacelike framed curves and the hyperbolic spinor $\phi$ represents the triad $\{\nu_1,\nu_2, \mu\}$.
The derivative of the spinor $\phi$ can be written by using the curvatures with respect to the adapted spacelike framed frame (according to the Frenet-Serret type frame):
\begin{equation}\label{spadapted2}
    \cfrac{d\phi}{ds}=\cfrac{1}{2}\left[j\left(  q+\Theta^\prime \right)\phi+pe^{j\Theta}\widehat\phi\right].
\end{equation}
\end{theorem}
\begin{proof}
From the equations \eqref{new}, \eqref{qq}, \eqref{qqq}, we have:
\begin{align*}
     \cfrac{d\phi}{ds}&=\cfrac{1}{2}\left[jl_1\phi-\left(l_2+jl_3\right)\widehat\phi\right]\\
     &=\cfrac{1}{2}\left[j\left(q+\Theta^\prime \right)\phi-\left(p\cosh\Theta+jp\sinh\Theta\right)\widehat\phi\right]\\
   % & =\cfrac{1}{2}\left[j\left( q+\Theta^\prime \right)\phi-p\left(\cosh\Theta+j\sinh\Theta\right)\widehat\phi\right]\\
    & =\cfrac{1}{2}\left[j\left(  q+\Theta^\prime \right)\phi-pe^{j\Theta}\widehat\phi\right].
\end{align*}
\end{proof}

\begin{corollary}  Let $\left(\gamma, \nu_1,\nu_2 \right), \left(\gamma, \breve\nu_1,\breve\nu_2 \right): I\rightarrow\mathbb{R}_1^3\times\Delta_1$ be spacelike framed curves and the hyperbolic spinors $\phi$ and $\Upsilon$ represent the triads $\{\nu_1,\nu_2, \mu\}$ and $\{\breve\nu_1,\breve\nu_2,\mu   \}$, respectively. The angle between $\cfrac{d\phi}{ds}$ and $\widehat{\phi}$ is $\Theta$ on condition that $\Theta^\prime=0$.
\end{corollary}

\begin{theorem} Let $\left(\gamma, \breve\nu_1,\breve\nu_2 \right): I\rightarrow\mathbb{R}_1^3\times\Delta_1$ be a spacelike framed curves and the hyperbolic spinor $\Upsilon$ represents the triad $\{\breve\nu_1,\breve\nu_2, \mu\}$.
The derivative of the spinor $\Upsilon$ can be written by using the curvatures with respect to the adapted spacelike framed frame (according to the Bishop type frame):
\begin{equation*}
    \cfrac{d\Upsilon}{ds}=-\cfrac{1}{2}\left(l_2\cosh\Theta-l_3\sinh\Theta  +j\left(-l_2\sinh\Theta+l_3\cosh\Theta\right)\right)\widehat\Upsilon.
\end{equation*}
\end{theorem}

\begin{theorem} Let $\left(\gamma, \breve\nu_1,\breve\nu_2 \right): I\rightarrow\mathbb{R}_1^3\times\Delta_1$ be a spacelike framed curves and the hyperbolic spinor $\Upsilon$ represents the triad $\{\breve\nu_1,\breve\nu_2, \mu\}$.
The derivative of the spinor $\Upsilon$ can be written by using the curvatures with respect to the adapted spacelike framed frame (according to the Frenet-Serret type frame):
\begin{equation*}
    \cfrac{d\Upsilon}{ds}=\cfrac{1}{2}\left(j\left(-\Theta^\prime+l_1\right)\Upsilon-||\mu^\prime||\widehat\Upsilon\right).
\end{equation*}
\end{theorem}

\begin{theorem} Let $\left(\gamma, \nu_1,\nu_2 \right), \left(\gamma, \breve\nu_1,\breve\nu_2 \right): I\rightarrow\mathbb{R}_1^3\times\Delta_1$ be spacelike framed curves and the hyperbolic spinors $\phi$ and $\Upsilon$ represent the triads $\{\nu_1,\nu_2, \mu\}$ and $\{\breve\nu_1,\breve\nu_2,\mu   \}$, respectively. Then the following equation is satisfied:
\begin{equation*}
    \cfrac{d\phi}{ds}=\pm\cfrac{1}{2}e^{j\frac{\Theta}{2}}\left[j\left( q+\Theta^\prime \right)\Upsilon-p\widehat\Upsilon\right].
\end{equation*}
\end{theorem}
\begin{proof}
By using the equation \eqref{3.12}, \eqref{4.17} and \eqref{spadapted2}, we have the desired result.
\end{proof}

\subsection{Hyperbolic spinor representations of timelike framed curves}

\quad

In this section, we obtain the hyperbolic spinors representations of timelike framed curves in the three-dimensional Minkowski space $\mathbb{R}_1^3$.

One should note that in this part of this study, we do not give the proofs for the sake of brevity, since all proofs can be shown using the methods used in the previous subpart.

\begin{definition} Let
$\left(\gamma,  {\nu}_1 ,  {\nu}_2  \right): I\rightarrow \mathbb{R}_1^3\times\Delta_2$ be a timelike framed curve and the hyperbolic spinor $\psi$ corresponds to the triad $\{  {\mu},\nu_1, {\nu_2}\}$. Then, the hyperbolic spinor representations of the timelike framed curve are written as follows:
\begin{align*}
   \mu +j\nu_1&=\psi^t\sigma\psi,
\\
\nu_2&=-\widehat{\psi}^t\sigma\psi,
\end{align*}
where $\overline{\psi}^t\psi=1$.
\end{definition}

\begin{theorem}\label{thmnew3}
Let $\left(\gamma, \nu_1,\nu_2 \right): I\rightarrow\mathbb{R}_1^3\times\Delta_1$ be a timelike framed curve and the hyperbolic spinor $\psi$ corresponds to the triad $\{\mu,\nu_1, \nu_2\}$. The single spinor equation that includes the curvatures of timelike framed curve is written as:
\begin{equation*}\label{newtimelike}
    \cfrac{d\psi}{ds}=\cfrac{1}{2}\left[jl_2\psi-\left(l_3+jl_1\right)\widehat\psi\right].
\end{equation*}
\end{theorem}
Similar to the previous section, we do not give hyperbolic spinor representations of the adapted frame of the timelike framed curve because of the fact that all theorems and corollaries in this part with respect to the hyperbolic spinor representations of the adapted frame of the timelike framed curve can be constructed easily. However, we want to present the following definition:
\begin{definition} Let
$\left(\gamma,  \breve{\nu}_1 ,  \breve{\nu}_2  \right): I\rightarrow \mathbb{R}_1^3\times\Delta_2$ be a timelike framed curve and the hyperbolic spinor $\Upsilon$ corresponds to the triad $\{{\mu}, \breve\nu_1, \breve{\nu}_2\}$. Then, the hyperbolic spinor representations of the adapted frame along the timelike framed curve are given as follows:
\begin{align*}
  \mu+j\breve\nu_1&=\eta^t\sigma\eta,
\\
\breve\nu_2&=-\widehat{\eta}^t\sigma\eta,
\end{align*}
where $\overline{\eta}^t\eta=1$.
\end{definition}

In addition to these, the following single spinor equations with respect to the adapted frame along the curve timelike framed curve can be given:
\begin{enumerate}
    \item 
    Let $\left(\gamma, \breve\nu_1,\breve\nu_2 \right): I\rightarrow\mathbb{R}_1^3\times\Delta_2$ be a timelike framed curve and the hyperbolic spinor $\eta$ represents the triad $\{{\mu}, \breve\nu_1, \breve{\nu}_2\}$. Then, we get (according to the Bishop type frame): 
    \begin{equation}
    \cfrac{d\eta}{ds}=\cfrac{1}{2}\left(  j\breve l_2\eta-\breve l_3\widehat\eta \right).
    \end{equation}
    \item
    Let $\left(\gamma, \breve\nu_1,\breve\nu_2 \right): I\rightarrow\mathbb{R}_1^3\times\Delta_2$ be a timelike framed curve and the hyperbolic spinor $\eta$ represents the triad $\{{\mu}, \breve\nu_1, \breve{\nu}_2\}$. Then, we have (according to the Frenet-Serret type frame): 
    \begin{equation}\label{spacelikeadapted}
    \cfrac{d\eta}{ds}=\cfrac{j}{2}\left(p\eta-q\widehat\eta \right).
\end{equation}
\end{enumerate}

\begin{theorem}\label{thmnew2timelike} Let $\left(\gamma, \nu_1,\nu_2 \right): I\rightarrow\mathbb{R}_1^3\times\Delta_2$ be a timelike framed curve and the hyperbolic spinor $\psi$ corresponds to the triad $\{\mu,\nu_1, \nu_2\}$. The hyperbolic spinor representations of timelike framed vectors are given as:
\begin{align*}
 {\mu}&=\cfrac{1}{2}\left(\psi^t\sigma\psi-{\widehat{\psi}^t\sigma\widehat{\psi}}\right),\\
  \nu_1 &=\cfrac{j}{2}\left(\psi^t\sigma\psi+\widehat{\psi}^t\sigma\widehat\psi\right), \\
      {\nu_2}&=-\widehat{\psi}^t\sigma\psi.
\end{align*}
\end{theorem}

\begin{corollary}
   Let $\left(\gamma, \nu_1,\nu_2 \right): I\rightarrow\mathbb{R}_1^3\times\Delta_2$ be a timelike framed curve and the hyperbolic spinor $\psi$ corresponds to the triad $\{ \mu,\nu_1,\nu_2\}$. Then the hyperbolic spinor components for timelike framed vectors are given as:
    \begin{align*}
      \mu  =& \cfrac{1}{2}\left(\psi^2_1 -\psi^2_2+\overline\psi^2_1-\overline\psi^2_2, j\left(\psi^2_1+\psi^2_2-\overline\psi^2_1-\overline\psi^2_2   \right), -2\left(\psi_1\psi_2+\overline\psi_1\overline\psi_2\right) \right),
\\
            \nu_1 =& \cfrac{j}{2}\left(\psi^2_1 -\psi^2_2-\overline\psi^2_1+\overline\psi^2_2, j\left(\psi^2_1+\psi^2_2+\overline\psi^2_1+\overline\psi^2_2   \right), 2\left(\overline\psi_1\overline\psi_2-\psi_1\psi_2\right) \right),
             \vspace{+2mm}\\
             \nu_2=& \left(\psi_1\overline\psi_2 +\overline\psi_1\psi_2,j\left(\psi_1\overline\psi_2-\overline\psi_1\psi_2\right),|\psi_1|^2-|\psi_2|^2   \right).
    \end{align*}
\end{corollary}

\begin{theorem} Let $\left(\gamma, \nu_1,\nu_2 \right): I\rightarrow\mathbb{R}_1^3\times\Delta_2$ be a timelike framed curve and the hyperbolic spinor $\psi$ corresponds to the triad $\{\mu,\nu_1, \nu_2\}$.
The derivative of the spinor $\psi$ can be written by using the curvatures with respect to the adapted frame of the timelike framed curve (according to the Bishop type frame):
\begin{equation*}
    \cfrac{d\psi}{ds}=\cfrac{1}{2}\left[j\left(\cos\Theta\breve l_2+\sin\Theta \breve l_3  \right)\psi-\left(\left(-\sin\Theta \breve l_2+\cos\theta \breve l_3\right)+j\Theta^\prime\right)\widehat\psi \right].
\end{equation*}
\end{theorem}

\begin{theorem} Let $\left(\gamma, \nu_1,\nu_2 \right): I\rightarrow\mathbb{R}_1^3\times\Delta_2$ be a timelike framed curve and the hyperbolic spinor $\psi$ corresponds to the triad $\{\mu,\nu_1, \nu_2\}$.
The derivative of the spinor $\psi$ can be written by using the curvatures with respect to the adapted frame of the timelike framed curve (according to the Frenet-Serret type frame):
\begin{equation*}
    \cfrac{d\psi}{ds}=\cfrac{1}{2}\left[jp\cos\Theta\psi-\left(-p\sin\Theta+j\left(q+\Theta^\prime\right)\right)\widehat\psi\right].
\end{equation*}
\end{theorem}

\begin{theorem} Let $\left(\gamma, \breve\nu_1,\breve\nu_2 \right): I\rightarrow\mathbb{R}_1^3\times\Delta_2$ be a timelike framed curve and the hyperbolic spinor $\eta$ corresponds to the triad $\{\mu,\breve\nu_1, \breve\nu_2\}$.
The derivative of the spinor $\eta$ (according to the Bishop type frame) can be expressed by using the curvatures with respect to the general frame of the timelike framed curve as follows:
\begin{equation*}
    \cfrac{d\eta}{ds}=\cfrac{1}{2}\left[j\left(l_2\cos\Theta -l_3\sin\Theta \right)\eta-\left(l_2\sin\Theta+l_3\cos\Theta\right)\widehat\eta\right].
\end{equation*}
\end{theorem}

\begin{theorem} Let $\left(\gamma, \breve\nu_1,\breve\nu_2 \right): I\rightarrow\mathbb{R}_1^3\times\Delta_2$ be a timelike framed curve and the hyperbolic spinor $\eta$ corresponds to the triad $\{\mu,\breve\nu_1, \breve\nu_2\}$.
The derivative of the spinor $\eta$ (according to the Frenet-Serret-type) can be given by using the curvatures with respect to the general frame of the timelike framed curve as follows:
\begin{equation*}
    \cfrac{d\eta}{ds}=\cfrac{j}{2}\left(||\mu^\prime||\eta-\left(-\Theta^\prime+l_1\right)\widehat\eta\right).
\end{equation*}
\end{theorem}

\section{Examples}\label{examples}
In this section, we construct two examples with respect to the hyperbolic spinor representations of both spacelike and timelike framed curves.
\begin{example}
Let us take spacelike curve $\gamma$ in $\mathbb{R}_1^3$ which is defined by 
\begin{equation}\label{curve}
    \gamma(s)=\left(\cfrac{s^3}{3},\cfrac{s^4}{4}+\cfrac{s^3}{3}, \cfrac{s^5}{5}+\cfrac{s^3}{3}  \right).
\end{equation}
The singular point of $\gamma(s)$ is $s=0$. Then, we have:
\begin{align*}
   \nu_1(s)&=\cfrac{1}{\sqrt{\left(  s^2+1 \right)^2+\left( s+1 \right)^2}}\left(0, s^2+1, -s-1 \right),\\
      \nu_2(s)&=\cfrac{1}{\sqrt{\left(s^4+3s^2+2s+1 \right)\left(s^4+3s^2+2s+2\right)}}\left(\left( s^2+1 \right)^2+\left(s+1\right)^2, s+1, s^2+1 \right),\\
       \mu(s)&=\cfrac{-1}{\sqrt{s^4+3s^2+2s+1 }}\left(1, s+1, s^2+1 \right),
       \\
       \delta(s)&=1.
\end{align*}
We can easily say that $\left(\gamma,\nu_1,\nu_2\right):I\rightarrow\mathbb{R}_1^3\times\Delta_1$ is a spacelike framed curve.
In the following Figure \ref{fig:1}, we can examine the spacelike framed curve given in the equation \eqref{curve}.
\begin{figure}[h!]
\centering
\includegraphics[width=1.8in]{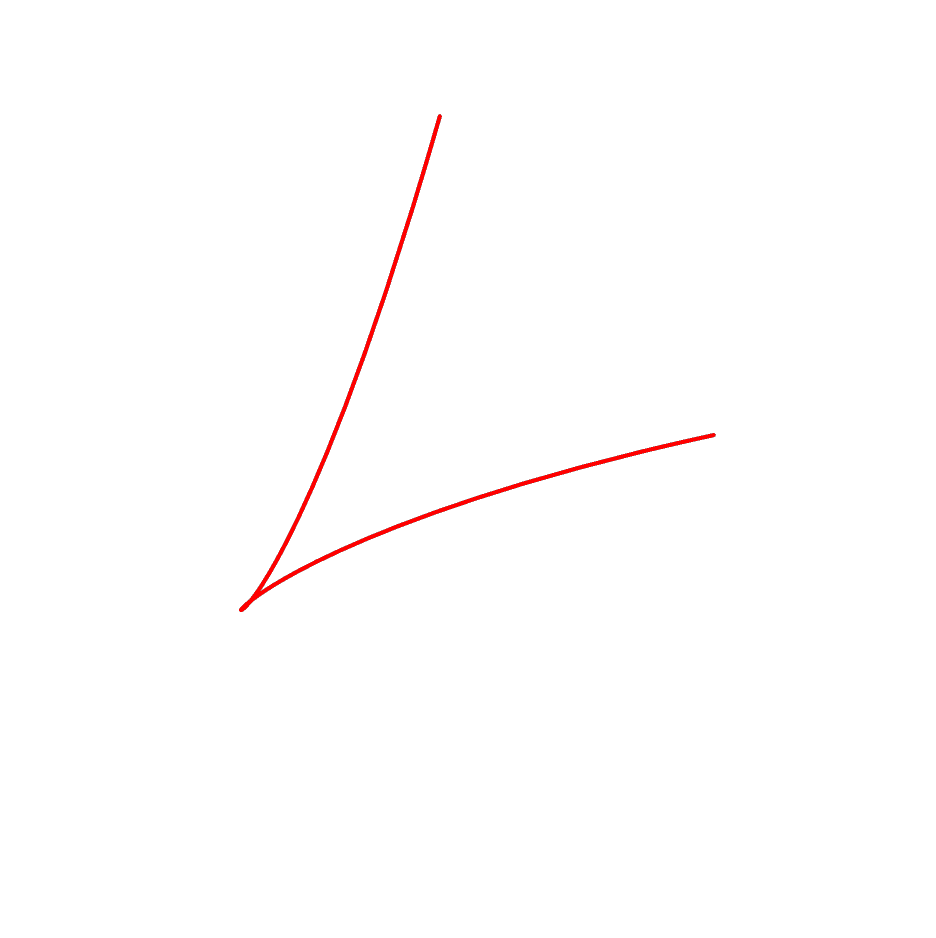} \caption{The spacelike framed curve $\gamma(s)$}
\label{fig:1}
\end{figure}

Then, by straightforward calculations, we have the curvature of the spacelike framed curve as follows:
\begin{align*}
   l_1(s)&=\cfrac{s \left(s^4+2 s^3+2 s^2+6 s-3\right)}{\left(s^4+3 s^2 +2 s+   2\right)^{3/2} \sqrt{s^8+6 s^6+4 s^5+12 s^4+12 s^3+13 s^2+6 s+2}},\\
   l_2(s)&=-\cfrac{s^2+2s-1}{\sqrt{\left(s^4+3 s^2+2 s+1\right)\left(s^4+3 s^2+2 s+2\right)}},\\
   l_3(s)&=\cfrac{2 s^3+3 s+1}{\sqrt{\left(s^4+3 s^2+2 s+1\right) \left(s^8+6 s^6+4 s^5+12 s^4+12 s^3+13 s^2+6 s+2\right)}},\\
   \alpha(s)&=-s^2\sqrt{s^4+3s^2+2s+1}.
\end{align*}
In addition to these, the following hyperbolic spinor components can be given:
\begin{align*}
    \phi_1&=\pm\cfrac{\sqrt{s+1+j\left(\left(s^2+1\right)^2+\left(s+1\right)^2+\left(s^2+1\right)\sqrt{s^4+3s^2+2s+2}\right)}}{\sqrt[4]{4\left(s^4+3s^2+2s+1\right)\left(s^4+3s^2+2s+2   \right)}},\\
     \phi_2&=\pm\cfrac{\sqrt{s+1+j\left(-\left(s^2+1\right)^2-\left(s+1\right)^2+\left(s^2+1\right)\sqrt{s^4+3s^2+2s+2}\right)}}{\sqrt[4]{4\left(s^4+3s^2+2s+1\right)\left(s^4+3s^2+2s+2   \right)}}.
\end{align*}
By using the equation \eqref{new}, we have:
\footnotesize
\begin{equation*} 
\cfrac{d\phi}{ds}=\cfrac{1}{2}\left[\begin{array}{rl}
&j\left(\cfrac{s \left(s^4+2 s^3+2 s^2+6 s-3\right)}{\left(s^4+3 s^2 +2 s+   2\right)^{3/2} \sqrt{s^8+6 s^6+4 s^5+12 s^4+12 s^3+13 s^2+6 s+2}}\right)\phi\\-
   &\left(\begin{array}{rl}&-\cfrac{s^2+2s-1}{\sqrt{\left(s^4+3 s^2+2 s+1\right)\left(s^4+3 s^2+2 s+2\right)}}\\&+j\left(\cfrac{2 s^3+3 s+1}{\sqrt{\left(s^4+3 s^2+2 s+1\right) \left(s^8+6 s^6+4 s^5+12 s^4+12 s^3+13 s^2+6 s+2\right)}} \right)  \end{array} \right)\widehat\phi\end{array}\right].
\end{equation*}
\normalsize
\end{example}

\begin{example}
Let us take timelike curve $\gamma$ in $\mathbb{R}_1^3$ which is defined by 
\begin{equation}\label{curve2}
    \gamma(s)=\left(2\left(s-3   \right)\sinh{s}-2\cosh{s},-2\left(   s-3 \right)\cosh{s}+2\sinh{s},\cfrac{s^2}{2}-3s  \right).
\end{equation}
The singular point of $\gamma(s)$ is $s=0$. Also, we have:
\begin{align*}
   \nu_1(s)&=\left( \sinh s, -\cosh s,0  \right),\\
    \nu_2(s)&=\left( -\cfrac{1}{\sqrt{3}}\cosh s,\cfrac{1}{\sqrt{3}}\sinh s,-\cfrac{2}{{\sqrt{3}}}    \right),\\
    \mu(s)&=\left(  \cfrac{2}{\sqrt{3}}\cosh s,-\cfrac{2}{\sqrt{3}}\sinh s,\cfrac{1}{\sqrt{3}} \right).
\end{align*}
In that case, we can easily say that $\left(\gamma,\nu_1,\nu_2\right):I\rightarrow\mathbb{R}_1^3\times\Delta_2$ is a timelike framed curve.
In the following Figure \ref{fig:2}, we can see the timelike framed curve presented in the equation \eqref{curve2}.
\begin{figure}[h!]
\centering
\includegraphics[width=1.8in]{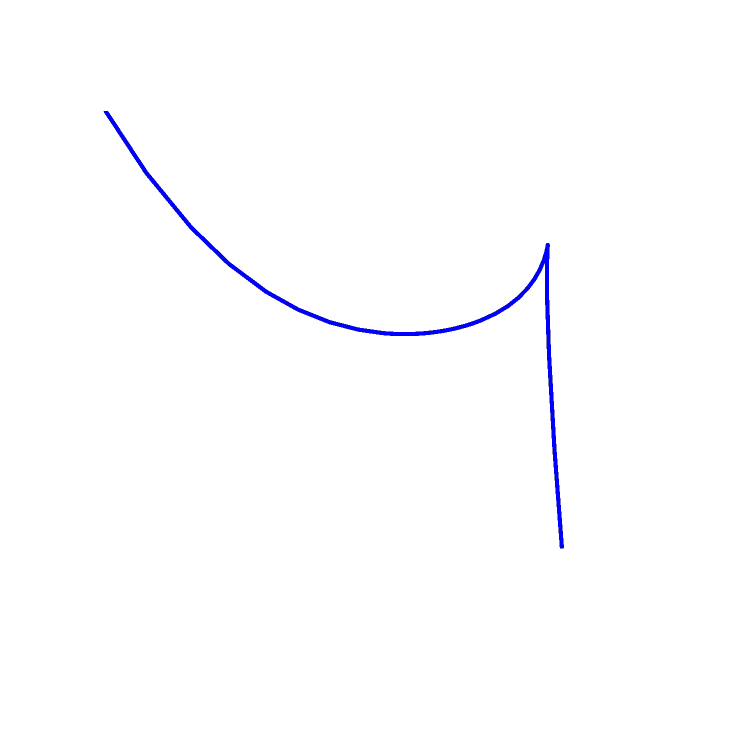} \caption{The timelike framed curve $\gamma(s)$}
\label{fig:2}
\end{figure}
\newpage
\noindent
Then, by straightforward calculations, we have the curvature of the timelike framed curve as follows:
\begin{align*}
   p(s)&=\cfrac{2}{\sqrt{3}},\\
   q(s)&=\cfrac{1}{\sqrt{3}},\\
    \alpha(s)&=\sqrt{3}\left(s-3\right).
\end{align*}
Additionally, the following hyperbolic spinor components can be given:
\begin{align*}
    \eta_1&=\pm\cfrac{1}{\sqrt{2}}\sqrt{\left(-1+\cfrac{2}{\sqrt{3}}\right)e^{-js}},\\
    \eta_2&=\pm\cfrac{1}{\sqrt{2}}\sqrt{\left(-1-\cfrac{2}{\sqrt{3}}\right)e^{js}}.
\end{align*}
With the help of the equation \eqref{spacelikeadapted}, we obtain:
\begin{equation*}
    \cfrac{d\eta}{ds}=\cfrac{j}{2}\left( \cfrac{2}{\sqrt{3}}\eta-\cfrac{1}{\sqrt{3}}\widehat\eta  \right).
\end{equation*}
\end{example}

\section{Conclusions}\label{conclusions}
The main purpose of this study is to investigate the hyperbolic spinor representations of not only spacelike framed curves but also timelike framed curves in Minkowski 3-space. In order to make the theorems and materials we found in this study more comprehensive, and to contribute to the literature, first, adapted frames for spacelike and timelike framed curves were obtained. We found hyperbolic spinor representations for non-null framed curves and some relations using both the existing information in the literature and the newly produced adapted frames. Also, we obtained the hyperbolic spinor equations of vectors of non-null framed curves, and then components of these vectors. We also construct some geometric interpretations and results with respect to these new hyperbolic spinor representations. Then, we construct numerical examples in order to construct given theorems and results with some useful figures.

We strongly believe that this study will make a great contribution to the existing literature for both classical differential geometry, singularity theory and also spinor structure, and we hope that it will shed light on future studies.

We plan that in the future study, we investigate the hyperbolic spinor representations of spacelike and timelike framed curves in de Sitter 2-space in detail.

\section*{Acknowledgement} First,  our deceased co-author Assist. Prof. Dr. Mehmet G\"uner for his contributions to education, science, and all his students. Also, the authors would like to thank Prof. Dr. Murat Tosun for his valuable suggestions and support. Moreover, the authors would like to thank editors and anonymous referees for their invaluable comments and careful reading.

\section*{Declarations}
\bmhead{Funding} Not applicable.
\bmhead{Conflict of interest} Not applicable.
\bmhead{Ethics approval} Not applicable.
\bmhead{Consent to participate} Not applicable.
\bmhead{Consent for publication} Not applicable.
\bmhead{Availability of data and materials} Consent for publication.
\bmhead{Code availability}  Consent for publication.
\bmhead{Authors' contributions} All authors contributed equally.

%\bibliographystyle{...}
%\bibliography{...}

\end{document}